\documentclass[12pt]{amsart}
\usepackage{amssymb}
\usepackage{calrsfs}
\usepackage{url}
\newcommand{\Z}{{\mathbb{Z}}}
\newcommand{\Q}{{\mathbb{Q}}}

\newcommand{\R}{{\mathbb{R}}}
\newcommand{\C}{{\mathbb{C}}}

\newcommand{\slm}{{\mathfrak{sl}}}
\newcommand{\fg}{{\mathfrak{g}}}
\newcommand{\fh}{{\mathfrak{h}}}

\newcommand{\be}{{\mathbf{e}}}

\newcommand{\cB}{{\mathcal{B}}}
\newcommand{\tcB}{{\tilde{\mathcal{B}}}}

\newcommand{\ovi}{{\underline{i}}}

\newcommand{\oalpha}{{\underline{\alpha}}}
\newcommand{\obeta}{{\underline{\beta}}}

\newcommand{\tfg}{{\tilde{\mathfrak{g}}}}
\newcommand{\tfh}{{\tilde{\mathfrak{h}}}}
\newcommand{\talpha}{{\tilde{\alpha}}}
\newcommand{\tbeta}{{\tilde{\beta}}}
\newcommand{\tgamma}{{\tilde{\gamma}}}
\newcommand{\tha}{{\tilde{h}}}
\newcommand{\te}{\tilde{e}}
\newcommand{\tf}{\tilde{f}}
\newcommand{\tq}{\tilde{q}}
\newcommand{\tI}{{\tilde{I}}}
\newcommand{\tPhi}{{\tilde{\Phi}}}
\newcommand{\tbe}{\tilde{\mathbf{e}}}

\newcommand{\heta}{{\hat{\eta}^\epsilon}}

\newcommand{\hgt}{{\operatorname{ht}}}
\newcommand{\sgn}{{\operatorname{sgn}}}

\renewcommand{\leq}{\leqslant}
\renewcommand{\geq}{\geqslant}

\newtheorem{thm}{Theorem}[section]
\newtheorem{lem}[thm]{Lemma}
\newtheorem{prop}[thm]{Proposition}
\newtheorem{cor}[thm]{Corollary}

\theoremstyle{definition}
\newtheorem{defn}[thm]{Definition}
\newtheorem{exmp}[thm]{Example}
\newtheorem{abs}[thm]{}

\theoremstyle{remark}
\newtheorem{rem}[thm]{Remark}

\numberwithin{equation}{section}

\begin{document}
\title{Canonical structure constants for simple Lie algebras}
\author{Meinolf Geck and Alexander Lang}
\address{Lehrstuhl f\"ur Algebra\\Universit\"at Stuttgart\\
Pfaffenwaldring 57\\D--70569 Stuttgart\\ Germany}
\curraddr{}
\email{meinolf.geck@mathematik.uni-stuttgart.de}
\email{alexander@lang-stb.de}
\thanks{This work is a contribution to the SFB-TRR 195 ``Symbolic Tools in
Mathematics and their Application'' of the German Research Foundation 
(DFG)}


\date{}

\begin{abstract}
Let $\fg$ be a finite-dimensional simple Lie algebra over $\C$. In the 
1950s Chevalley showed that $\fg$ admits particular bases, now called 
``Chevalley bases'', for which the corresponding structure constants are 
integers. Such bases are not unique but, using Lusztig's theory of canonical
bases, one can single out a ``canonical'' Chevalley basis which is unique 
up to a global sign. In this paper, we give explicit formulae for
the structure constants with respect to such a basis.
\end{abstract}

\maketitle

\section{Introduction} \label{sec0}

Let $\fg$ be a finite-dimensional simple Lie algebra over $\C$. We shall
assume that the reader is familiar with the basic aspects of the 
Cartan--Killing structure theory for $\fg$; see, e.g., Bourbaki 
\cite{bour2} or Humphreys \cite{H}. Let $\fh\subseteq \fg$ be a Cartan
subalgebra and $\fh^*$ be the dual space. For any $\lambda \in \fh^*$
let $\fg_\lambda$ be the subspace of all $x\in \fg$ such that $[h,x]=
\lambda(h)x$ for $h\in \fh$. Then the root system of $\fg$ is the set of 
all non-zero $\lambda\in \fh^*$ such that $\fg_\lambda \neq \{0\}$. We have 
the fundamental Cartan decompositon 
\[ \fg=\fh\oplus \bigoplus_{\alpha \in \Phi} \fg_\alpha \quad\mbox{where}
\quad \mbox{$\fh=\fg_0\;\;$ and $\;\;\dim \fg_\alpha=1\;$ 
($\alpha \in \Phi$)}.\]
Let us choose elements $0\neq e_\alpha \in \fg_\alpha$ for all $\alpha \in 
\Phi$. Then $\{e_\alpha\mid \alpha \in \Phi\}$, together with a basis of 
$\fh$, yields a basis of $\fg$. In order to completely describe the 
multiplication table for $\fg$, the main problem are the Lie brackets
$[e_\alpha,e_\beta]$ where $\alpha,\beta\in \Phi$. If $\beta=-\alpha$, then
$[e_\alpha, e_{-\alpha}]=c_\alpha h_\alpha$, where $0\neq c_\alpha \in \C$ 
and $h_\alpha\in \fh$ is the co-root corresponding to~$\alpha$. Now there 
is a partition of $\Phi$ into positive and negative roots. By fixing 
elements $e_\alpha$ for all positive roots $\alpha \in \Phi$, one can 
always choose $e_{-\alpha}$ such that $c_\alpha$ takes any desired value. 
(Usually one takes $c_\alpha=1$ but we will see that other choices are
equally valuable.) Furthermore, it is known that $[e_\alpha,e_\beta]=0$ 
if $\alpha+\beta\not \in \Phi$ (e.g., if $\alpha=\beta$); otherwise, we have 
\[ [e_\alpha,e_\beta]=N_{\alpha,\beta}e_{\alpha+\beta} \qquad \mbox{where}
\qquad 0\neq N_{\alpha,\beta} \in \C.\] 
The constants $N_{\alpha,\beta}$ are the main object of interest in this 
paper. So let now $\alpha,\beta \in \Phi$ be such that $\beta\neq\pm\alpha$.
Then we set 
\[p_{\alpha,\beta}:=\max\{i \geq 0 \mid \beta+i\alpha \in \Phi\}\quad
\mbox{and}\quad q_{\alpha,\beta}:=\max\{i \geq 0 \mid \beta-i\alpha \in 
\Phi\}.\]
Thus, the sequence $\beta-q_{\alpha,\beta}\alpha,\ldots,\beta-\alpha,\beta,
\beta+\alpha,\ldots, \beta+p_{\alpha,\beta}\alpha$ is the ``$\alpha$-string
through~$\beta$''. All elements in that sequence are roots in $\Phi$. 
Chevalley \cite{chev} showed that one can always choose the elements 
$e_\alpha$ such that 
\[ N_{\alpha,\beta}=\pm (q_{\alpha,\beta}+1) \in \Z\qquad \mbox{if $\alpha,
\beta,\alpha+\beta \in \Phi$}.\]
In this case, we say that $\{e_\alpha\mid \alpha\in \Phi\}$ (together with
a suitable basis of $\fh$ determined by that selection) is a ``Chevalley 
basis'' of $\fg$. Note that such bases are not unique. One can replace each 
$e_\alpha$ individually by $\pm e_\alpha$, and one still obtains a Chevalley
basis. Now there is Lusztig's theory \cite{L1} of canonical bases for 
quantised enveloping algebras. This gives rise to canonical bases in all 
finite-dimensional irreducible $\fg$-modules. In particular, there is a 
canonical basis of $\fg$, viewed as a $\fg$-module via the adjoint 
representation; see \cite{L5} and further references there to previous work 
of Lusztig on this subject (prior to \cite{L1}). In \cite{mylie}, this is 
presented in a completely elementary way, without reference to the general 
theory in \cite{L1} or \cite{L5}. The result is that, by an inductice 
procedure entirely within $\fg$ itself, one can single out a ``canonical'' 
collection of the elements $\{e_\alpha\mid \alpha \in \Phi\}$, which is 
unique up to replacing all $e_\alpha$ simultaneously by $-e_\alpha$. 
Hence, the resulting structure constants $N_{\alpha,\beta}$ are also 
uniquely determined up to a global sign. The purpuse of this paper is to 
find explicit formulae for them.

In Section~\ref{sec1}, we recall some further facts from the structure
theory of simple Lie algebras and define the ``canonical'' collection
of the elements $e_\alpha$ mentioned above. In Section~\ref{sec2} we
give an explicit formula for the correspondong ``canonical'' structure 
constants $N_{\alpha,\beta}$, in the case where $\fg$ is simply laced, 
that is, of type $A_n$ ($n\geq 1$), $D_n$ ($n\geq 3$) or $E_n$ ($n=6,7,8$). 
The proof of the main result, Theorem~\ref{thm1}, is by a general argument, 
not ``case--by--case''. In Section~\ref{sec3}, we use the technique of 
``folding'' to obtain analogous results for $\fg$ of type $B_n$, $C_n$ 
($n\geq 2$), $G_2$ and~$F_4$. 

That technique is well-known in Lie theory; see, e.g, De Graaf 
\cite[\S 5.15]{graaf} or Kac \cite[\S 7.9]{kac}. See also Mitzman 
\cite{mitz} where this is used to construct integral forms of type
$2$ and $3$ affine Lie algebras and their universal enveloping 
algebras. Lusztig \cite{L1} uses ``admissible automorphisms'' ($=$ 
``folding'' in the finite case) to reduce the study of canonical bases 
for quantised enveloping algebras in the non-symmetric case to the 
symmetric case. In Section~\ref{sec3} we will present the ``folding'' 
procedure for $\fg$ in a way where the canonical basis of $\fg$ is 
built in from the outset; this may be of independent interest. In 
particular, ``folding'' will be seen to be perfectly compatible with 
canonical bases. 

\section{The $\epsilon$-canonical Chevally basis of $\fg$} \label{sec1}

Let $\fg$ be a simple Lie algebra, $\fh\subseteq \fg$ be a Cartan 
subalgebra and $\Phi\subseteq \fh^*$ be the root system, as in 
Section~\ref{sec0}. Let $I$ be a finite index set and $\Pi=\{\alpha_i
\mid i \in I\}$ be a set of simple roots of $\Phi$. Every $\alpha\in 
\Phi$ can be written uniquely as $\alpha=\sum_{i \in I} n_i \alpha_i$ 
where $n_i\in \Z$ and either $n_i\geq 0$ for all $i \in I$ or $n_i\leq 0$ 
for all $i \in I$. Correspondingly, we say that $\alpha$ is a positive 
or a negative root, respectively. Thus, $\alpha$ can be represented by 
an $I$-tuple $(n_i)_{i \in I}$ where all $n_i$ have the same sign. The 
integer $\hgt(\alpha):= \sum_{i \in I} n_i$ is called the height of~$\alpha$.

\begin{abs} \label{coroot} For every $\alpha \in \Phi$ let $h_\alpha\in
\fh$ be the corresponding co-root. This is characterised as follows. It 
is known that $[\fg_\alpha,\fg_{-\alpha}]$ is a $1$-dimensional subspace 
of $\fh$, and that $\alpha$ is non-zero on $[\fg_\alpha,\fg_{-\alpha}]$. 
Then $h_\alpha$ is the unique element in that subspace on which~$\alpha$ 
takes the value~$2$. (See, e.g., Bourbaki \cite[Ch.~VIII, \S 2, 
no.~2]{bour2}.) If we now take any $0\neq e_\alpha\in \fg_\alpha$, then 
there is an element $e_{-\alpha} \in \fg_{-\alpha}$ such that $[e_\alpha,
e_{-\alpha}]=\pm h_\alpha$. Thus, the subspace $\langle e_\alpha,
e_{-\alpha}, h_\alpha \rangle_\C\subseteq \fg$ is a subalgebra isomorphic 
to $\slm_2$. (For example, in \cite{bour2}, it is assumed that $[e_\alpha,
e_{-\alpha}]= -h_\alpha$ for all $\alpha$; in \cite{H}, it is assumend that 
$[e_\alpha,e_{-\alpha}]= +h_\alpha$. For us it will be convenient to keep 
the flexibility of choosing a sign for each $\alpha\in \Phi$.) 
\end{abs}

\begin{abs} \label{coroot1} The root system $\Phi$ has the following 
invariance property (which gives rise to the Weyl group of $\fg$ but we 
will not need to formally introduce this). Then  
\[\langle \alpha,\beta\rangle:=\beta(h_\alpha)\in \Z \qquad\mbox{and} 
\qquad \beta- \langle \alpha,\beta\rangle \alpha \in \Phi\qquad \mbox{for 
all $\alpha,\beta\in \Phi$}.\]
This has the following consequence. Assume that $m:=\beta(h_\alpha)\neq 0$. 
Then the $\alpha$-string through $\beta$ contains the term $\beta-m\alpha$. 
Hence, if $m<0$, then we will also have $\beta+\alpha\in \Phi$; if $m>0$, 
then we will also have $\beta-\alpha\in \Phi$. (For all this see, e.g.,
Bourbaki \cite[Ch.~VI, \S 1, no.~3]{bour1} or \cite[Ch.~VIII, \S 2, 
no.~2]{bour2}.)
\end{abs} 

\begin{abs} \label{eps1}
We set $h_i:=h_{\alpha_i}$ for $i\in I$. We have just seen that there 
are elements $e_i \in \fg_{\alpha_i}$ and $f_i \in \fg_{-\alpha_i}$ such 
that $h_i = [e_i,f_i]$ for all $i\in I$. It is known that $\fg$ is 
generated (as a Lie algebra) by the elements $e_i,f_i$ ($i\in I$). We 
call $\{e_i,f_i,h_i \mid i \in I\}$ a system of \textit{Chevalley 
generators} for~$\fg$. (By the Isomorphism Theorem in \cite[\S 14.2]{H}, any 
two such systems can be transformed into each other by an automorphisms 
of $\fg$.) The corresponding Cartan matrix is defined by $A=(a_{ij})_{i,j
\in I}$ where 
\[ a_{ij}:=\langle \alpha_i,\alpha_j\rangle =\alpha_j(h_i)\in \Z \quad 
\mbox{for all $i,j\in I$}.\]
Thus, we have $[h_i,e_j]=a_{ij}e_j$ and $[h_i,f_j]=-a_{ij}f_j$ for all
$i,j\in I$. We also have $[e_i,f_j]=0$ for $i\neq j$ in $I$.
\end{abs}

Now it is known that $A=(a_{ij})_{i,j\in I}$ is an 
indecomposable (generalised) Cartan matrix of finite type (see 
\cite[Chap.~4]{kac}). These matrices are encoded by the Dynkin diagrams 
in Table~\ref{Mdynkintbl}, as follows. The vertices of the diagram are 
labelled by the elements of~$I$. Let $i,j\in I$, $i\neq j$. If $a_{ij}=
a_{ji}=0$, then there is no edge between the vertices labelled by~$i$ 
and~$j$. Now assume that $a_{ij} \neq 0$. Then we also have $a_{ji}\neq 0$ 
and the notation can be chosen such that $a_{ij}=-1$ and $m:=-a_{ji}\in 
\{1,2,3\}$. In this case, the vertices labelled by $i$, $j$ will be joined 
by $m$ edges; if $m\geq 2$, then we put an additional arrow pointing 
towards~$j$.

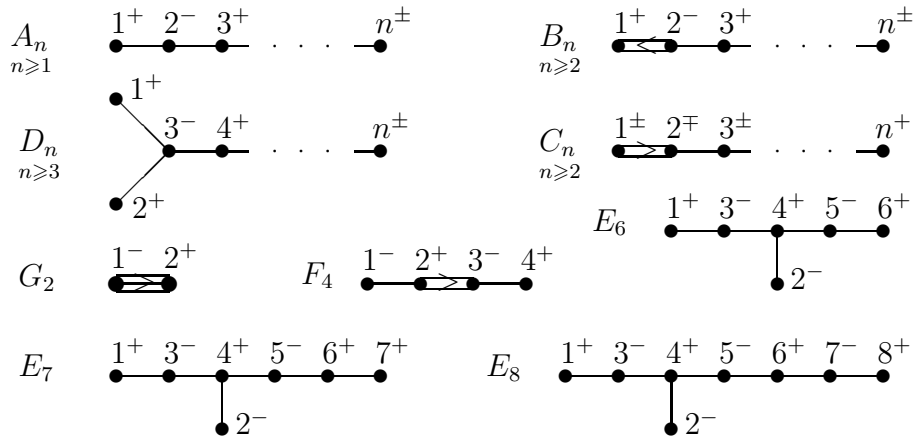
\begin{table}[htbp] \caption{Dynkin diagrams of Cartan matrices of 
finite  type} \label{Mdynkintbl} 
\begin{center} 
\makeatletter
\begin{picture}(345,170)
\put( 13, 25){$E_7$}
\put( 50, 25){\circle*{5}}
\put( 48, 30){$1^+$}
\put( 50, 25){\line(1,0){20}}
\put( 70, 25){\circle*{5}}
\put( 68, 30){$3^-$}
\put( 70, 25){\line(1,0){20}}
\put( 90, 25){\circle*{5}}
\put( 88, 30){$4^+$}
\put( 90, 25){\line(0,-1){20}}
\put( 90, 05){\circle*{5}}
\put( 95, 03){$2^-$}
\put( 90, 25){\line(1,0){20}}
\put(110, 25){\circle*{5}}
\put(108, 30){$5^-$}
\put(110, 25){\line(1,0){20}}
\put(130, 25){\circle*{5}}
\put(128, 30){$6^+$}
\put(130, 25){\line(1,0){20}}
\put(150, 25){\circle*{5}}
\put(148, 30){$7^+$}

\put(190, 25){$E_8$}
\put(220, 25){\circle*{5}}
\put(218, 30){$1^+$}
\put(220, 25){\line(1,0){20}}
\put(240, 25){\circle*{5}}
\put(238, 30){$3^-$}
\put(240, 25){\line(1,0){20}}
\put(260, 25){\circle*{5}}
\put(258, 30){$4^+$}
\put(260, 25){\line(0,-1){20}}
\put(260, 05){\circle*{5}}
\put(265, 03){$2^-$}
\put(260, 25){\line(1,0){20}}
\put(280, 25){\circle*{5}}
\put(278, 30){$5^-$}
\put(280, 25){\line(1,0){20}}
\put(300, 25){\circle*{5}}
\put(298, 30){$6^+$}
\put(300, 25){\line(1,0){20}}
\put(320, 25){\circle*{5}}
\put(318, 30){$7^-$}
\put(320, 25){\line(1,0){20}}
\put(340, 25){\circle*{5}}
\put(338, 30){$8^+$}

\put( 13, 59){$G_2$}
\put( 50, 60){\circle*{6}}
\put( 48, 66){$1^-$}
\put( 50, 57){\line(1,0){20}}
\put( 50, 60){\line(1,0){20}}
\put( 50, 63){\line(1,0){20}}
\put( 56, 57.5){$>$}
\put( 70, 60){\circle*{6}}
\put( 68, 66){$2^+$}

\put(120, 60){$F_4$}
\put(145, 60){\circle*{5}}
\put(143, 65){$1^-$}
\put(145, 60){\line(1,0){20}}
\put(165, 60){\circle*{5}}
\put(163, 65){$2^+$}
\put(165, 58){\line(1,0){20}}
\put(165, 62){\line(1,0){20}}
\put(171, 57.5){$>$}
\put(185, 60){\circle*{5}}
\put(183, 65){$3^-$}
\put(185, 60){\line(1,0){20}}
\put(205, 60){\circle*{5}}
\put(203, 65){$4^+$}

\put(230, 80){$E_6$}
\put(260, 80){\circle*{5}}
\put(258, 85){$1^+$}
\put(260, 80){\line(1,0){20}}
\put(280, 80){\circle*{5}}
\put(278, 85){$3^-$}
\put(280, 80){\line(1,0){20}}
\put(300, 80){\circle*{5}}
\put(298, 85){$4^+$}
\put(300, 80){\line(0,-1){20}}
\put(300, 60){\circle*{5}}
\put(305, 58){$2^-$}
\put(300, 80){\line(1,0){20}}
\put(320, 80){\circle*{5}}
\put(318, 85){$5^-$}
\put(320, 80){\line(1,0){20}}
\put(340, 80){\circle*{5}}
\put(338, 85){$6^+$}

\put( 13,110){$D_n$}
\put( 13,100){$\scriptstyle{n \geq 3}$}
\put( 50,130){\circle*{5}}
\put( 55,130){$1^+$}
\put( 50,130){\line(1,-1){21}}
\put( 50, 90){\circle*{5}}
\put( 56, 85){$2^+$}
\put( 50, 90){\line(1,1){21}}
\put( 70,110){\circle*{5}}
\put( 68,115){$3^-$}
\put( 70,110){\line(1,0){30}}
\put( 90,110){\circle*{5}}
\put( 88,115){$4^+$}
\put(110,110){\circle*{1}}
\put(120,110){\circle*{1}}
\put(130,110){\circle*{1}}
\put(140,110){\line(1,0){10}}
\put(150,110){\circle*{5}}
\put(147,115){$n^\pm$}

\put(210,110){$C_n$}
\put(210,100){$\scriptstyle{n \geq 2}$}
\put(240,110){\circle*{5}}
\put(238,115){$1^{\pm}$}
\put(240,108){\line(1,0){20}}
\put(240,112){\line(1,0){20}}
\put(246,107.5){$>$}
\put(260,110){\circle*{5}}
\put(258,115){$2^{\mp}$}
\put(260,110){\line(1,0){30}}
\put(280,110){\circle*{5}}
\put(278,115){$3^{\pm}$}
\put(300,110){\circle*{1}}
\put(310,110){\circle*{1}}
\put(320,110){\circle*{1}}
\put(330,110){\line(1,0){10}}
\put(340,110){\circle*{5}}
\put(337,115){$n^{+}$}

\put( 10,150){$A_n$}
\put( 10,140){$\scriptstyle{n \geq 1}$}
\put( 50,150){\circle*{5}}
\put( 48,155){$1^+$}
\put( 50,150){\line(1,0){20}}
\put( 70,150){\circle*{5}}
\put( 68,155){$2^-$}
\put( 70,150){\line(1,0){30}}
\put( 90,150){\circle*{5}}
\put( 88,155){$3^+$}
\put(110,150){\circle*{1}}
\put(120,150){\circle*{1}}
\put(130,150){\circle*{1}}
\put(140,150){\line(1,0){10}}
\put(150,150){\circle*{5}}
\put(148,155){$n^{\pm}$}

\put(210,150){$B_n$}
\put(210,140){$\scriptstyle{n \geq 2}$}
\put(240,150){\circle*{5}}
\put(238,155){$1^+$}
\put(240,148){\line(1,0){20}}
\put(240,152){\line(1,0){20}}
\put(246,147.5){$<$}
\put(260,150){\circle*{5}}
\put(258,155){$2^-$}
\put(260,150){\line(1,0){30}}
\put(280,150){\circle*{5}}
\put(278,155){$3^+$}
\put(300,150){\circle*{1}}
\put(310,150){\circle*{1}}
\put(320,150){\circle*{1}}
\put(330,150){\line(1,0){10}}
\put(340,150){\circle*{5}}
\put(338,155){$n^\pm$}
\end{picture}
\end{center}
\end{table}

\begin{abs} \label{dynk1} In the diagrams in Table~\ref{Mdynkintbl} we 
also specify a function $\epsilon\colon I \rightarrow \{\pm 1\}$ such that 
$\epsilon(i)=-\epsilon(j)$ whenever $i\neq j$ and $a_{ij}\neq 0$. Note 
that, since the diagram is connected, there are exactly two such functions:
if $\epsilon$ is one of them, then the other one is~$-\epsilon$. (The
conventions for the definition of $\epsilon$ for types $B_n$, $C_n$, $G_2$,
$F_4$ will be explained in Section~\ref{sec3}.) This function is an 
essential ingredient in the definition of a ``canonical Chevalley basis''
of $\fg$ below.
\end{abs}

\begin{abs} \label{chevb} The following remarks will be useful in 
identifying a Chevalley basis for $\fg$. Let us choose, for each $\alpha
\in \Phi$, elements $e_\alpha\in \fg_\alpha$ and $e_{-\alpha}\in
\fg_{-\alpha}$ such that $[e_\alpha,e_{-\alpha}]=\pm h_\alpha$. Furthermore, 
let $\omega\colon \fg \rightarrow \fg$ be the unique automorphism such that
$\omega(e_i)=f_i$, $\omega(f_i)=e_i$ and $\omega(h_i)=-h_i$ for $i\in I$; 
note that $\omega^2=\mbox{id}_\fg$. (This exists by the Isomorphism 
Theorem \cite[\S 14.2]{H}, or see \cite[Remark~4.8]{mylie}.) Assume that
\[\omega(e_\alpha)=\pm e_{-\alpha} \qquad \mbox{for all $\alpha\in \Phi$}.\]
Now let $\alpha,\beta\in \Phi$ be such that $\alpha+\beta\in \Phi$. Then, by 
Chevalley \cite[p.~23]{chev}, we have $N_{\alpha,\beta}N_{-\alpha,-\beta}=
\pm (q_{\alpha,\beta}+1)^2$. Using the automorphism $\omega$, it follows that 
$N_{-\alpha,-\beta}=\pm N_{\alpha,\beta}$ and, hence,
\begin{equation*}
N_{\alpha,\beta}^2=\pm (q_{\alpha,\beta}+1)^2\qquad \mbox{whenever}
\qquad \alpha,\beta,\alpha+\beta\in \Phi.\tag{$*$}
\end{equation*}
So if we know for some reason that $N_{\alpha,\beta}\in \R$, then
$N_{\alpha,\beta}=\pm (q_{\alpha,\beta}+1)$. Assume now that this is
the case; then $B:=\{h_i \mid i \in I\}\cup \{e_\alpha\mid \alpha\in\Phi\}$ 
is a \textit{Chevalley basis} of~$\fg$. Since each $h_\alpha$ is known 
to be a $\Z$-linear combination of $h_i$ ($i\in I$), it follows that the 
complete multiplication table of $\fg$ with respect to $B$ has only
entries in~$\Z$. Thus, a Chevalley basis defines an ``integral structure'' 
of~$\fg$.
\end{abs}

\begin{abs} \label{cb4} We can now describe Lusztig's \textit{canonical 
basis} of $\fg$, in the elementary version of \cite{mylie}. Having fixed 
$\epsilon\colon I\rightarrow \{\pm 1\}$ (see \ref{dynk1}), there is a 
unique collection of elements $\{0\neq \be_\alpha^\epsilon \in 
\fg_\alpha \mid \alpha\in \Phi\}$ such that the following relations 
hold, for any $i\in I$:
\begin{alignat*}{2} 
\be_{\alpha_i}^\epsilon &= \epsilon(i)e_i, &\quad \be_{-\alpha_i}^\epsilon
&= -\epsilon(i)f_i,\\
[e_i,\be_\alpha^\epsilon]&=(q_{\alpha_i,\alpha}+1)\be_{\alpha+
\alpha_i}^\epsilon &\qquad &\mbox{if $\alpha+\alpha_i\in \Phi$},\\
[f_i,\be_\alpha^\epsilon]&=(p_{\alpha_i,\alpha}+1)\be_{\alpha-
\alpha_i}^\epsilon &\qquad &\mbox{if $\alpha-\alpha_i\in \Phi$}.
\end{alignat*}
See \cite[\S 5]{mylie}. If we replace $\epsilon$ by $-\epsilon$, then 
$\be_\alpha^{-\epsilon}=-\be_\alpha^\epsilon$ for all $\alpha\in\Phi$. So
the passage from $\epsilon$ to $-\epsilon$ is given by a very simple rule.
In this setting, it automatically follows that 
\[\omega(\be_\alpha^\epsilon)=-\be_{-\alpha}^\epsilon \quad\mbox{and}
\quad [\be_\alpha^\epsilon,\be_{-\alpha}^\epsilon]=(-1)^{\hgt(\alpha)}
h_\alpha \quad \mbox{for all $\alpha\in \Phi$};\] 
see \cite[Theorem~5.7]{mylie}. We call
\[ \cB^\epsilon=\{h_i\mid i \in I\}\cup \{\be_\alpha^\epsilon\mid \alpha
\in\Phi\}\] 
the \textit{$\epsilon$-canonical Chevalley basis} of $\fg$. (The above 
rules provide an efficient algorithm for constructing $\cB^\epsilon$; see 
\cite{chevlie}.) The structure constants with respect to the collection 
$\{\be_\alpha^\epsilon\mid \alpha \in \Phi\}$ will be denoted by 
$N_{\alpha,\beta}^\epsilon$; thus we have 
\[ [\be_\alpha^\epsilon,\be_\beta^\epsilon] =N_{\alpha,\beta}^\epsilon
\be_{\alpha+\beta}^\epsilon \qquad \mbox{if $\alpha,\beta,\alpha+\beta
\in \Phi$}.\]
It will be convenient to set $N_{\alpha,\beta}^\epsilon:=0$ for any 
$\alpha,\beta\in \Phi$ such that $\alpha\neq \pm \beta$ and $\alpha+
\beta\not\in\Phi$. It is already known (see once more 
\cite[Theorem~5.7]{mylie}) that 
\[N_{\alpha,\beta}^\epsilon=\eta^\epsilon(\alpha,\beta)(q_{\alpha,\beta}
+1) \qquad \mbox{if $\alpha,\beta,\alpha+\beta\in \Phi$},\]
where $\eta^\epsilon(\alpha,\beta)=\pm 1$. So $\cB^\epsilon$ indeed is
a Chevalley basis. Our aim is to obtain explicit formulae for the signs 
$\eta^\epsilon(\alpha,\beta)$, just in terms of the roots $\alpha, 
\beta\in \Phi$ and~$\epsilon$.
\end{abs}

\begin{exmp}[Type $A_{n-1}$] \label{typeA} Let $n\geq 2$ and $\fg=\slm_n$ 
be the Lie algebra of $n\times n$-matrices with trace zero. Let $\fh
\subseteq \fg$ be the subalgebra consisting of all diagonal matrices 
with trace zero. It is known that $\fg$ is simple and $\fh$ is a Cartan
subalgebra of~$\fg$. For $1\leq i\leq n$, let $\delta_i\in \fh^*$ be the 
map which sends a diagonal matrix to its $i$th diagonal entry; note that 
$\delta_1+ \ldots+\delta_n=0$. For $i\neq j$ let $e_{ij} \in \fg$ be the 
matrix with entry~$1$ at position $(i,j)$, and $0$ everywhere else. Then 
a simple matrix calculation shows that
\[[h,e_{ij}]=(\delta_i(x)-\delta_j(x))e_{ij}
\qquad \mbox{for all $h\in \fh$},\]
and so $e_{ij}\in \fg_{\delta_i-\delta_j}$. It easily follows that
$\Phi=\{\delta_i-\delta_j \mid 1\leq i,j\leq n,i\neq j\}$.
Now set $\alpha_i:=\delta_{i}-\delta_{i+1}$ for $1\leq i
\leq n-1$. Then $\{\alpha_1,\ldots,\alpha_{n-1}\}$ is a system of simple 
roots for $\Phi$. For all $i\neq j$ we have
\[ \delta_i-\delta_j=\left\{\begin{array}{cc} 
\alpha_{i}+ \alpha_{i+1}+ \ldots +\alpha_{j-1} & \quad 
\mbox{if $i<j$},\\ -(\alpha_{j}+\alpha_{j+1}+\ldots + \alpha_{i-1}) & 
\quad \mbox{if $i>j$}. \end{array}\right.\]
Now let $I:=\{1,\ldots,n-1\}$; for $i\in I$ we set 
\[ h_i:=e_{ii}-e_{i+1,i+1}\in \fh, \qquad e_i:=e_{i,i+1}\in \fg_{\alpha_i},
\qquad f_i=e_{i+1,i}\in \fg_{-\alpha_i}.\]
Then $\{e_i,f_i,h_i\mid i \in I\}$ is a system of Chevalley generators for
$\fg$. Since the Lie brackets $[e_i,e_{kl}]$ are easily determined
for all $i$ and all $k\neq l$, one can directly verify in this case that
\begin{center}
\fbox{$\be_{\alpha}^\epsilon=-(-1)^{\hgt(\alpha)}\epsilon(i) e_{ij} 
\qquad \mbox{if $\alpha=\delta_i-\delta_j$, $i\neq j$}.$}
\end{center}
(For this purpose, one has to check that the relations in \ref{cb4} hold;
note that the absolute value of $\hgt(\alpha)$ equals $|i-j|$ if $\alpha=
\delta_i-\delta_j$.) Let $\alpha=\delta_i-\delta_j$ and $\beta=\delta_j-
\delta_k$ where $i,j,k$ are pairwise different. Then $\alpha+\beta=
\delta_i-\delta_k$ and 
\begin{center}
\fbox{$[\be_\alpha^\epsilon,\be_\beta^\epsilon]=-\epsilon(j)
\be_{\alpha+\beta}^\epsilon,\quad$ that is, $\quad N_{\alpha,
\beta}^\epsilon(\alpha,\beta)=-\epsilon(j)$.}
\end{center}
\end{exmp}

Similar explicit descriptions have been determined by the second named
author for all the classical Lie algebras $\fg$ of type $B_n$, $C_n$, 
$D_n$ in their natural matrix representations; see \cite[Chap.~2]{Lang}.

\section{The simply laced case} \label{sec2}

We keep the notation of the previous section. Our aim will be to describe
the structure constants $N_{\alpha,\beta}^\epsilon$ with respect to the 
$\epsilon$-canonical Chevalley basis 
\[ \cB^\epsilon= \{h_i\mid i \in I\} \cup \{\be_\alpha^\epsilon\mid
\alpha\in \Phi\}\] 
of $\fg$ (see \ref{cb4}). In this section, we shall deal with the case 
where $A$ is \textit{simply laced}, that is, $a_{ij}\in \{0,\pm 1\}$ for 
all $i\neq j$ in $I$. 

\begin{rem} \label{sec20} Let $\alpha,\beta\in \Phi$ be such that 
$\alpha+\beta\in \Phi$. By the defining properties of the Lie bracket, 
we certainly habe $N_{\beta,\alpha}^\epsilon=-N_{\alpha,\beta}^\epsilon$.
Now let $\omega\colon \fg\rightarrow \fg$ be the automorphism in 
\ref{chevb}. As already mentioned in \ref{cb4}, we have 
$\omega(\be_\alpha^\epsilon)=-\be_{-\alpha}^\epsilon$ for all $\alpha
\in \Phi$. This implies that $N_{-\alpha,-\beta}^\epsilon=-N_{\alpha,
\beta}^\epsilon$.  Hence, in order to determine the structure constants 
$N_{\alpha,\beta}^\epsilon$, it is sufficient to consider the case where 
$\alpha$ is a positive root (but~$\beta$ may still be an arbitrary root).
\end{rem}

\begin{abs} \label{ADE0} Assume that $A$ is simply laced. This is equivalent 
to saying that $A$ is symmetric, and also to saying that $A$ is of type $A_n$
($n\geq 1$), $D_n$ ($n\geq 3$) or $E_n$ ($n=6,7,8$). Assume now that this
is the case. Then the following hold, where $\alpha,\beta\in \Phi$ are
such that $\alpha\neq \pm \beta$.
\begin{itemize}
\item[(a)] $\;0\leq p_{\alpha, \beta}+q_{\alpha,\beta}\leq 1$ and
$\langle \alpha,\beta\rangle\in \{0,\pm 1\}$. 
\item[(b)] $\;\langle \alpha,\beta\rangle=1 \Leftrightarrow 
\alpha-\beta \in \Phi$, and $\langle \alpha,\beta\rangle=-1 
\Leftrightarrow \alpha+\beta \in \Phi$.
\item[(c)] $\;\langle \alpha,\beta\rangle=\langle \beta,\alpha\rangle$.
(This immediately follows from (a) and (b).)
\item[(d)] $\,$ If $\alpha+\beta\in \Phi$, then $q_{\alpha,\beta}=0$ 
and so $N_{\alpha,\beta}^\epsilon=\pm 1$.
\end{itemize} 
(See, e.g., Bourbaki \cite[Ch.~VI, \S 1, no.~3]{bour1} or the 
discussion in Carter \cite[\S 3.4]{C1}.)
\end{abs}

\begin{exmp} \label{heta0} Assume that $A$ is simply laced. Let $\alpha,
\beta\in \Phi$. Using the properties in \ref{ADE0}, we obtain the following 
formula, which will be useful below. Let $\alpha,\beta\in \Phi$. Write
$\alpha=\sum_{i\in I} n_i\alpha_i$ and $\beta=\sum_{j \in I} m_j \alpha_j$ 
where $n_i,m_j\in I$. Then 
\begin{align*}
\langle \alpha,\beta\rangle&=\beta(h_\alpha)=
\sum_{j \in I} m_j \alpha_j(h_\alpha)=
\sum_{j \in I} m_j \langle\alpha,\alpha_j\rangle=
\sum_{j \in I} m_j \langle\alpha_j,\alpha\rangle\\ &=
\sum_{j \in I} m_j \alpha(h_j)=\sum_{i,j \in I} n_im_j \alpha_i(h_j)=
\sum_{i,j \in I} n_im_j a_{ji}=\sum_{i,j \in I} a_{ij}n_im_j.
\end{align*}
\end{exmp}

\begin{lem} \label{lem1} Assume that $A$ is simply laced. Let $l\in I$ 
and $\alpha,\beta\in \Phi$ be such that $\alpha_l+\alpha\in \Phi$ and 
$\alpha_l+\alpha+\beta\in \Phi$. Also assume that $\alpha\neq \pm \beta$ and
$\beta \neq \pm \alpha_l$. Then we have either $\alpha+\beta\in \Phi$ or 
$\alpha_l+\beta\in \Phi$ (but not both). Accordingly, we have:
\[ N_{\alpha_l+\alpha,\beta}^\epsilon=\left\{\begin{array}{cl}
N_{\alpha,\beta}^\epsilon & \quad \mbox{if $\alpha+\beta \in \Phi$},\\
-N_{\alpha,\alpha_l+\beta}^\epsilon & \quad \mbox{if $\alpha_l+\beta\in 
\Phi$}.\end{array}\right.\]
\end{lem}

\begin{proof} We have $[\be_{\alpha_l}^\epsilon,\be_\alpha^\epsilon]
=N_{\alpha_l,\alpha}^\epsilon \be_{\alpha_l+\alpha}^\epsilon$,
$[\be_{\alpha_l+\alpha}^\epsilon,\be_\beta^\epsilon]
=N_{\alpha_l+\alpha,\beta}^\epsilon \be_{\alpha_l+\alpha+\beta}^\epsilon$.
So, using the Jacobi identity and the anti-symmetry of the Lie bracket,
we obtain
\begin{align*}
N_{\alpha_l,\alpha}^\epsilon N_{\alpha_l+\alpha,\beta}^\epsilon
\be_{\alpha_l+\alpha+\beta}^\epsilon &=[[\be_{\alpha_l}^\epsilon,
\be_\alpha^\epsilon],\be_\beta^\epsilon]=
-[[\be_{\alpha}^\epsilon, \be_\beta^\epsilon]],\be_{\alpha_l}^\epsilon]
-[[\be_{\beta}^\epsilon, \be_{\alpha_l}^\epsilon]],\be_{\alpha}^\epsilon]\\
&=[\be_{\alpha_l}^\epsilon,[\be_{\alpha}^\epsilon, \be_\beta^\epsilon]]
-[\be_\alpha^\epsilon,[\be_{\alpha_l}^\epsilon, \be_{\beta}^\epsilon]].
\end{align*}
Now, if $\alpha+\beta\in \Phi$, then $[\be_\alpha^\epsilon,
\be_\beta^\varepsilon]=N_{\alpha,\beta}^\epsilon\be_{\alpha+\beta}^\epsilon$;
otherwise, this is~$0$ (since $\alpha\neq \pm \beta$). Similarly, if 
$\alpha_l+\beta\in \Phi$, then $[\be_{\alpha_l},\be_\beta^\epsilon]=
N_{\alpha_l,\beta}^\epsilon \be_{\alpha +\beta}^\epsilon$; otherwise, this 
is~$0$ (since $\beta\neq \pm \alpha_l$). Hence, since $N_{\alpha_l,
\alpha}^\epsilon=\epsilon(l)$ (see \ref{cb4} and note that $q_{\alpha_l,
\alpha}=0$ since $A$ is simply laced), we obtain the following formula:
\[ \epsilon(l)N_{\alpha_l+\alpha,\beta}^\epsilon=
N_{\alpha,\beta}^\epsilon N_{\alpha_l,\alpha+\beta}^\epsilon - 
N_{\alpha_l,\beta}^\epsilon N_{\alpha,\alpha_l+\beta}^\epsilon.\]
Since $A$ is simply laced, all non-zero structure constants are $\pm 1$.
Since the left hand side of the above identity is non-zero, we conclude
that either $N_{\alpha,\beta}^\epsilon N_{\alpha_l,\alpha+\beta}^\epsilon=
\pm 1$ or $N_{\alpha_l,\beta}^\epsilon N_{\alpha,\alpha_l+\beta}^\epsilon=
\pm 1$ (but not both). Hence, either $\alpha+\beta\in \Phi$ or $\alpha_l+
\beta\in \Phi$ (but not both). Accordingly, $\epsilon(l)N_{\alpha_l+
\alpha,\beta}^\epsilon =N_{\alpha,\beta}^\epsilon N_{\alpha_l,\alpha+
\beta}^\epsilon\neq 0$ or $\epsilon(l)N_{\alpha_l+\alpha,\beta}^\epsilon=
- N_{\alpha_l,\beta}^\epsilon N_{\alpha,\alpha_l+\beta}^\epsilon\neq 0$. 
In the first case, $N_{\alpha_l,\alpha+\beta}^\epsilon=\epsilon(l)$; in 
the second case, $N_{\alpha_l,\beta}^\epsilon=\epsilon(l)$. This
yields the desired formulae.
\end{proof}  

\begin{defn}[Lang \protect{\cite[\S 3.2.2]{Lang}}] \label{def1}
For $\alpha\in \Phi$ we set $\sgn(\alpha):=1$ if $\alpha$ is positive,
and $\sgn(\alpha):=-1$ if $\alpha$ is negative. Let $\alpha,\beta\in \Phi$
be such that $\alpha+\beta\in \Phi$. Writing $\alpha=\sum_{i \in I} 
n_i\alpha_i$ and $\beta=\sum_{j \in I} m_j\alpha_j$ with $n_i,m_j\in \Z$,
we define
\begin{equation*}
\heta(\alpha,\beta):= \sgn(\alpha)\sgn(\beta)\sgn(\alpha+\beta) 
\prod_{i,j\in I} \epsilon(i)^{a_{ij}n_im_j}=\pm 1.\tag{$\clubsuit$}
\end{equation*}
Using the identity $\langle \alpha_i,\beta\rangle=\beta(h_i)=
\sum_{j \in I} \alpha_j(h_i)m_j=\sum_{j \in I} a_{ij}m_j$, we can 
re-write the above formula as
\[ \heta(\alpha,\beta)= \sgn(\alpha)\sgn(\beta)\sgn(\alpha+\beta)
\prod_{i \in I} \epsilon(i)^{n_i\langle \alpha_i,\beta\rangle},\]
which will be useful in some arguments below.
\end{defn}

\begin{exmp} \label{def1a} Assume that $A$ is simply laced. Let $i \in I$ 
and $\beta\in \Phi$ be such that $\alpha_i+\beta\in \Phi$. We claim that 
$\heta(\alpha_i,\beta)= \epsilon(i)$. Indeed, we have $n_i=1$ and $n_j=0$
for $i\neq j$. It is known that $\beta$ and $\alpha_i+\beta$ have the same 
sign and so $\sgn(\beta)\sgn(\alpha_i+\beta)=1$. Hence, since $\sgn
(\alpha_i)=1$, we obtain $\heta(\alpha_i,\beta)=\epsilon(i)^{\langle 
\alpha_i,\beta \rangle}$. Since $\alpha_i+\beta\in \Phi$ we have $\langle 
\alpha_i,\beta \rangle=-1$ by \ref{ADE0}(b). So $\heta(\alpha_i,\beta)=
\epsilon(i)$. 
\end{exmp}

In the following two results we show that $\heta(\alpha,\beta)$ satisfies 
relations analogous to those in Remark~\ref{sec20} and Lemma~\ref{lem1}.

\begin{lem} \label{heta1} Assume that $A$ is simply laced.
If $\alpha,\beta,\alpha+\beta\in \Phi$, then
$\heta(\beta,\alpha)=\heta(-\alpha,-\beta)=-\heta(\alpha,\beta)$.
\end{lem}

\begin{proof} We have $\sgn(-\alpha)=-\sgn(\alpha)$, $\sgn(\-\beta)=-
\sgn(\beta)$ and $\sgn(-\alpha-\beta)=-\sgn(\alpha+\beta)$. Changing
$\alpha$ to $-\alpha$ and $\beta$ to $-\beta$ does not change anything
in the product over $i,j\in I$ in the definition of $\heta(\alpha,\beta)$.
Hence, we certainly have $\heta(-\alpha,-\beta)=-\heta(\alpha,\beta)$.
Now consider the relation between $\heta(\alpha,\beta)$ and $\heta(\beta,
\alpha)$. Let $c:=\sgn(\alpha)\sgn(\beta)\sgn(\alpha+\beta)$. If 
$i,j\in I$ are such that $a_{ij}=a_{ji}\neq 0$, then $\epsilon(i)=
-\epsilon(j)$. Consequently, we can re-write $\heta(\alpha, \beta)$ as 
\[\heta(\alpha,\beta)= c\prod_{i,j\in I} (-\epsilon(j))^{a_{ji}m_jn_i} 
=\heta(\beta,\alpha)\prod_{i,j\in I} (-1)^{a_{ij}n_im_j}.\]
By Example~\ref{heta0}, the right hand side equals $\heta(\beta,\alpha)
(-1)^{\langle \alpha,\beta\rangle}$. Since $\alpha+\beta\in \Phi$,
we have $\langle \alpha,\beta\rangle=-1$ by \ref{ADE0}(b). Hence,
$\heta(\alpha,\beta)= -\heta(\beta,\alpha)$.
\end{proof}

\begin{lem} \label{heta2} Assume that $A$ is simply laced. Let $l\in I$ 
and $\alpha,\beta\in \Phi$ be such that $\alpha_l+\alpha\in \Phi$ and 
$\alpha_l+\alpha+\beta\in \Phi$. Also assume that $\alpha\neq \beta$ and
$\beta\neq \pm \alpha_l$. By Lemma~\ref{lem1}, we have either $\alpha+\beta
\in \Phi$ or $\alpha_l+\beta\in \Phi$ (but not both). Then 
\[ \heta(\alpha_l+\alpha,\beta)=\left\{\begin{array}{cl}
\heta(\alpha,\beta) & \quad \mbox{if $\alpha+\beta \in \Phi$},\\
-\heta(\alpha,\alpha_l+\beta) & \quad \mbox{if $\alpha_l+\beta\in 
\Phi$}.\end{array}\right.\]
\end{lem}

\begin{proof} Let $c:=\sgn(\alpha)\sgn(\beta)\sgn(\alpha+ \beta)$. 
Since $\alpha$ and $\alpha_l+\alpha$ have the same sign, and $\alpha+\beta$ 
and $\alpha_l+\alpha+\beta$ have the same sign, we conclude that 
\[ \heta(\alpha_l+\alpha,\beta)=c\,\epsilon(l)^{(n_l+1)\langle \alpha_l,
\beta \rangle}\prod_{i\in I\setminus \{l\}} \epsilon(i)^{n_i\langle 
\alpha_i,\beta\rangle}=\epsilon(l)^{\langle \alpha_l,\beta \rangle}
\heta(\alpha,\beta).\]
Assume first that $\alpha+\beta\in \Phi$. Since $\alpha_l+\beta\not\in 
\Phi$, we have $\langle \alpha_l,\beta\rangle=0$ by \ref{ADE0}. Hence, 
$\heta(\alpha_l+\alpha,\beta\rangle=\heta(\alpha,\beta)$, as claimed. Now 
assume that $\alpha_l+\beta\in \Phi$. Since $\langle \alpha_l,\beta\rangle=
-1$ by \ref{ADE0}, we conclude that $\heta(\alpha_l+\alpha,\beta)=
\epsilon(l) \heta(\alpha,\beta)$. On the other hand, since $\beta$ and 
$\alpha_l+\beta$ have the same sign, we obtain 
\[ \heta(\alpha,\alpha_l+\beta)=c\,\prod_{i\in I} \epsilon(i)^{n_i
\langle \alpha_i,\alpha_l+\beta\rangle}=\heta(\alpha,\beta)
\prod_{i\in I} \epsilon(i)^{\langle \alpha_i,\alpha_l\rangle}. \]
Let $c':=\sgn(\alpha)\sgn(\alpha_l)\sgn(\alpha_l+\alpha)$. Then the product 
on the right hand side equals $c'\, \heta(\alpha,\alpha_l)$. Since $\alpha,
\alpha_l+\alpha$ have the same sign and since $\sgn(\alpha_l)=1$, 
we have $c'=1$. Furthermore, by Lemma~\ref{heta1} and Example~\ref{def1a}, 
we have $\heta(\alpha,\alpha_l)=-\heta(\alpha_l,\alpha)=-\epsilon(l)$. 
Hence, $\heta(\alpha, \alpha_l+\beta)=-\epsilon(l)\heta(\alpha,\beta)=-
\heta(\alpha_l+\alpha,\beta)$, as claimed.
\end{proof}

\begin{thm} \label{thm1} Assume that $A$ is simply laced. Let $\alpha,
\beta\in \Phi$ be such that $\alpha+\beta\in \Phi$. Then 
$[\be_\alpha^\epsilon, \be_\beta^\epsilon]=\heta(\alpha, \beta) 
\be_{\alpha+ \beta}$ with $\heta(\alpha,\beta)$ as in ($\clubsuit$).
\end{thm}

\begin{proof} We must show that $N_{\alpha,\beta}^\epsilon=\heta(\alpha,
\beta)$ whenever $\alpha,\beta,\alpha+\beta\in \Phi$. By Remark~\ref{sec20}
and Lemma~\ref{heta1}, it is sufficient to consider the case where 
$\alpha$ is a positive root (but $\beta$ may be arbitrary). Assume now 
that $\alpha\in \Phi^+$ ($=$ set of positive roots). Then we proceed by 
induction on $\hgt(\alpha)$. If $\hgt(\alpha)=1$, then $\alpha=\alpha_i$ 
for some $i\in i$. In this case, the assertion holds by Example~\ref{def1a} 
and the formulae in \ref{cb4}. Now let $\hgt(\alpha)>1$. It is well-known 
that we can find some $l\in I$ such that $\gamma:=\alpha-\alpha_l\in 
\Phi^+$. Then $\alpha_l+\gamma\in \Phi$ and $\alpha_l+ \gamma+\beta=
\alpha+\beta\in \Phi$. Assume first that $\gamma\neq \pm \beta$ and
$\beta \neq \pm \alpha_l$. Then we can apply Lemma~\ref{lem1} and obtain
\[ N_{\alpha,\beta}^\epsilon=N_{\alpha_l+\gamma,\beta}^\epsilon=
\left\{\begin{array}{cl} N_{\gamma,\beta}^\epsilon & \quad
\mbox{if $\gamma+\beta \in \Phi$},\\ -N_{\gamma,\alpha_l+\beta}^\epsilon 
& \quad \mbox{if $\alpha_l+\beta\in \Phi$}.\end{array}\right.\]
By induction, we already know that $N_{\gamma,\beta}^\epsilon=
\heta(\gamma,\beta)$ in the first case, and $N_{\gamma,\alpha_l+
\beta}^\epsilon= \heta(\gamma,\alpha_l+\beta)$ in the second case.
Using Lemma~\ref{heta2}, we conclude that the result equals
$\heta(\alpha,\beta)$ in both cases, as desired.

It remains to deal with the cases where $\gamma=\pm \beta$ or $\beta=
\pm \alpha_l$. If $\gamma=\beta$, then $\alpha_l+\gamma\in \Phi$ and
$\alpha_l+2\gamma= \alpha_l+\gamma+\beta\in \Phi$, contradiction to 
\ref{ADE0}(a). On the hand, if $\gamma=-\beta$, then 
\begin{alignat*} {2}
N_{\alpha,\beta}^\epsilon&=N_{\alpha,-\gamma}^\epsilon=
N_{\gamma,-\alpha}^\epsilon &\qquad& \mbox{(by Remark~\ref{sec20})},\\
\heta(\alpha,\beta)&=\heta(\alpha,-\gamma)=
\heta(\gamma,-\alpha) & \qquad & \mbox{(by Lemma~\ref{heta1})}.
\end{alignat*}
Now note that, by induction, the two right hand sides are equal; hence, so
are the left hand sides. Next, if $\beta=\alpha_l$, then $\gamma+\alpha_l
\in \Phi$ and $\gamma+2\alpha_l=\alpha+\beta\in \Phi$, contradiction to
\ref{ADE0}(a). On the other hand, if $\beta=-\alpha_l$, then 
\begin{alignat*} {2}
N_{\alpha,\beta}^\epsilon&=N_{\alpha,-\alpha_l}^\epsilon=N_{\alpha_l,
-\alpha}^\epsilon &\qquad& \mbox{(again by Remark~\ref{sec20})},\\
\heta(\alpha,\beta)&=\heta(\alpha,-\alpha_l)=\heta(\alpha_l,-\alpha) & 
\qquad & \mbox{(again by Lemma~\ref{heta1})}.
\end{alignat*}
By \ref{cb4} and Example~\ref{def1a}, the two right hand sides are both
equal to~$\epsilon(l)$.
\end{proof}

\section{Folding} \label{sec3}

We keep the general notation of the previous sections. We assume from
now that the Cartan matrix $A$ is simply laced; thus, $A$ is of type 
$A_n$ ($n\geq 1$), $D_n$ ($n\geq 3$) or $E_n$ ($n=6,7,8$). Furthermore,
let $I\rightarrow I$, $i \mapsto i^\prime$, be a bijection such that
\begin{alignat*}{2}
a_{ij}&=a_{i^\prime j^\prime} &&\qquad \mbox{for all $i,j\in I$}.\tag{a}
\\
a_{ii'}& =0 &&\qquad \mbox{for all $i\in I$ such that $i'\neq i$}.\tag{b}
\end{alignat*}
The first condition means that $i\mapsto i^\prime$ corresponds to a 
symmetry of the Dynkin diagram of $A$; the second condition means that,
if $i\neq i'$, then the nodes labelled by $i$ and $i'$ are not connected
in the Dynkin diagram. Let $d\geq 1$ denote the order of the bijection
$i\mapsto i^\prime$ (as an element of the symmetric group on~$I$). The 
non-trivial possibilities are listed in Table~\ref{Mautos}. (The last 
column will be explained below.) Note that there is also a non-trivial 
symmetry of order $2$ for $A$ of type $A_{2n}$ ($n \geq 1$), but 
condition (b) is not satisfied in this case. By the Isomorphism Theorem
in \cite[\S 14.2]{H}, there is a Lie algebra automorphism $\tau\colon \fg
\rightarrow \fg$ such that 
\[ \tau(e_i)=e_{i^\prime}, \quad \tau(f_i)=f_{i^\prime} \quad\mbox{and}
\quad \tau(h_i)=h_{i^\prime}\quad \mbox{for all $i\in I$}.\]
Since the permutation $i\mapsto i^\prime$ ($i\in I$) has
order $d$, we also have $\tau^d=\mbox{id}_\fg$.
 
\begin{table}[htbp] \caption{Diagram automorphisms}
\label{Mautos} \begin{center} {\small 
$\renewcommand{\arraystretch}{1.2} \begin{array}{cccc} \hline \mbox{Type 
of $A$} & d & \mbox{orbits of $i\mapsto i^\prime$}  & \mbox{$\tilde{A}$} 
\\ \hline A_{2n-1} \;(n\geq 2) & 2 & \{n\},\{n{-}1,n{+}1\},\{n{-}2,n{+}2\},
\ldots, \{1,2n{-}1\} & C_n\\ D_{n+1} \; (n\geq 3) & 2 & \{1,2\},\{3\},\{4\},
\ldots, \{n{+}1\} & B_n\\ D_4 & 3 & \{3\},\{1,2,4\} & G_2\\
E_6 & 2 & \{2\},\{4\},\{3,5\},\{1,6\} & F_4\\\hline \end{array}$}
\end{center}
\end{table}

\begin{abs} \label{auto1} Let $\tI$ be a set of representatives for 
the orbits of the bijection $i\mapsto i^\prime$ of $I$. Note that we can 
choose $\tI$ such that the subdiagram (of the original diagram of~$A$)
formed by the nodes in $\tI$ is connected. (For example, this holds if 
we select the first element listed in each orbit in the 3rd column of 
Table~\ref{Mautos}.) For each $i\in I$, let $d_i$ be the size of the 
orbit of~$i$. (Since $d=1$ or $d$ is a prime, we have $d_i=1$ or 
$d_i=d$.) Now define 
\[ \tilde{A}=(\tilde{a}_{ij})_{i,j\in \tI}\qquad \mbox{where} \qquad
\tilde{a}_{ij}:=\left\{\begin{array}{cl} d_i a_{ij}  & \quad \mbox{if
$d_i>d_j=1$},\\ a_{ij} & \quad \mbox{otherwise}.\end{array}\right.\]
Then a case--by--case verification shows that $\tilde{A}$ (for $d>1$) is 
an indecomposable Cartan matrix of the type specified in the last column
of Table~\ref{Mautos}. If we order the rows and columns of $\tilde{A}$ 
according to the list of orbits in the 2nd column, then the diagram of 
$\tilde{A}$ corresponds exactly to the one with the same name 
in Table~\ref{Mdynkintbl}. 

Let $\epsilon \colon I\rightarrow \{\pm 1\}$ be the function defined
in Table~\ref{Mdynkintbl} for $A$ of type $A_{2n-1}$, $D_{n+1}$ and $E_6$.
Then we notice that $\epsilon$ is constant on the orbits of~$I$. So, by 
restriction, we obtain analogous functions $\tilde{\epsilon} \colon \tI 
\rightarrow \{\pm 1\}$ for $\tilde{A}$ of type $B_n$, $C_n$, $G_2$ 
and $F_4$, as specified in Table~\ref{Mdynkintbl}. 
\end{abs}

We will now show that the automorphism $\tau\colon \fg\rightarrow \fg$ is 
compatible with the various constructions in the previous sections. First 
note that $\tau$ restricts to the linear map $\tau_\fh\colon\fh\rightarrow 
\fh$ such that $h_i\mapsto h_{i^\prime}$ ($i\in I$). Let $\tau_\fh^*\colon 
\fh^*\rightarrow \fh^*$ be the contragredient dual map, that is, 
$\tau_\fh^*(\lambda)=\lambda \circ \tau_\fh^{-1}$ for $\lambda\in \fh^*$. 

\begin{lem} \label{auto2} For any $\alpha\in \Phi$, we have $\alpha^\prime
:= \tau_\fh^*(\alpha) \in \Phi$. The map $\alpha \mapsto \alpha^\prime$ is 
a permutation of $\Phi$ such that $\alpha_{j}^\prime=\tau_\fh^*(\alpha_j)=
\alpha_{j^\prime}$ for all $j\in I$. Since $\alpha\mapsto \alpha^\prime$
is linear, we have $\hgt(\alpha)=\hgt(\alpha^\prime)$.
\end{lem}

\begin{proof} Let $\alpha\in \Phi$ and $x\in \fg_\alpha$. Note that, by
the definition $\tau_\fh^*$, we have $\alpha^\prime(\tau(h))=
\tau_\fh^*(\alpha)(\tau(h))=\alpha(h)$ for $h\in \fh$, and so 
\[[\tau(h),\tau(x)]=\tau([h,x])=\tau(\alpha(h)x)=\alpha(h)\tau(x)
=\alpha^\prime(\tau(h))\tau(x).\]
Since $\fh=\{\tau(h)\mid h\in \fh\}$, this shows that $\tau(x)\in 
\fg_{\alpha^\prime}$. So, since $\fg_\alpha\neq \{0\}$, we also have 
$\fg_{\alpha^\prime}\neq \{0\}$ and, hence, $\alpha^\prime\in \Phi$. 
Finally, let $\alpha=\alpha_j$ for some $j\in I$. Then 
\[\alpha_j^\prime(h_{i^\prime})=\tau_\fh^*(\alpha_j)(h_{i^\prime})=
\alpha_j(\tau_\fh^{-1}(h_{i^\prime}))=\alpha_j(h_i)=a_{ij}=
a_{i^\prime j^\prime}=\alpha_{j^\prime}(h_{i^\prime})\]
for all $i\in I$. Since $\{h_{i^\prime}\mid i \in I\}$ is a basis of $\fh$, 
this yields $\alpha_j^\prime=\tau_\fh^*(\alpha_j)=\alpha_{j^\prime}$.
\end{proof}

\begin{prop} \label{prop34} Let $\cB^\epsilon=\{h_i \mid i \in I\}\cup 
\{\be_\alpha^\epsilon\mid \alpha\in \Phi\}$ be the $\epsilon$-canonical 
Chevalley basis of $\fg$ (as in \ref{cb4}). Then we have 
\[ \tau(\be_\alpha^\epsilon)=\be_{\alpha^\prime}^\epsilon\quad \mbox{and}
\quad \tau(h_\alpha)=h_{\alpha^\prime}\quad\mbox{for all $\alpha \in 
\Phi$}.\] 
Furthermore, $N_{\alpha,\beta}^\epsilon=N_{\alpha^\prime,
\beta^\prime}^\epsilon$ if $\alpha,\beta,\alpha+\beta\in \Phi$.
\end{prop}

\begin{proof} Since $[\be_{\alpha}^\epsilon,\be_{-\alpha}^\epsilon]=
(-1)^{\hgt(\alpha)}h_\alpha$, it suffices to prove the assertion about
$\tau(\be_{\alpha}^\epsilon)$. First assume that $\alpha \in \Phi^+$. 
Then we proceed by induction on the height $\hgt(\alpha)$. If 
$\hgt(\alpha)=1$, then $\alpha=\alpha_i$ for some $i\in I$. By inspection 
of Table~\ref{Mautos}, we see that $\epsilon(i)=\epsilon(i^\prime)$. So we 
obtain
\[\tau(\be_{\alpha_i}^\epsilon)=\tau\bigl(\epsilon(i)e_i\bigr)=
\epsilon(i)\tau(e_i)=\epsilon(i)e_{i^\prime}=\epsilon(i)\epsilon(i^\prime)
\be_{\alpha_{i^\prime}} =\be_{\alpha_i^\prime}^\epsilon,\]
as required. Now let $\hgt(\alpha)>1$. Then, as is well--known, there exists
some $i\in I$ such that $\beta:=\alpha-\alpha_i\in \Phi^+$. By induction, we 
already know that $\tau(\be_\beta^\epsilon)=\be_{\beta^\prime}^\epsilon$. 
We have $[e_i,\be_\beta^\epsilon]=(q_{\alpha_i,\beta}+1)
\be_{\alpha}^\epsilon$. Since $A$ is simply laced, we have $q_{\alpha,
\beta}=0$; see \ref{ADE0}(a). This yields 
\[ \tau(\be_{\alpha}^\epsilon)=\tau\bigl([e_i,\be_\beta^\epsilon]\bigr)
=[\tau(e_i), \tau(\be_\beta^\epsilon)]=[e_{i^\prime}, 
\be_{\beta^\prime}^\epsilon].\] 
Since $\alpha_{i^\prime}+\beta^\prime=\tau_\fh^*(\alpha_i)+\tau_\fh^*(\beta)=
\tau_\fh^*(\alpha_i+\beta)=\tau_\fh^*(\alpha)=\alpha^\prime\in\Phi$, the 
right hand side of the above identity equals $(q_{\alpha_{i^\prime},
\beta^\prime}+1)\be_{\alpha^\prime}^\epsilon$. But, again, we have
$q_{\alpha_{i^\prime},\beta^\prime}=0$ and so $\tau(\be_{\alpha}^\epsilon)
=\be_{\alpha^\prime}^\epsilon$, as desired. 

Thus, the assertion holds for all $\alpha\in\Phi^+$. In order to deal
with negative roots, we consider the automorphism $\omega\colon \fg 
\rightarrow \fg$ in \ref{chevb}; we have $\omega(e_i)=f_i$, $\omega(f_i)=
e_i$ and $\omega(h_i)=-h_i$ for $i\in I$. One easiliy sees that $\omega$ 
and $\tau$ commute with each other. By \ref{cb4}, we have 
$\omega(\be_\alpha^\epsilon)=-\be_{-\alpha}^\epsilon$ for all $\alpha 
\in \Phi$. Now let $\alpha\in \Phi^-$. Then $-\alpha\in \Phi^+$ 
and so we already know that $\tau(\be_{-\alpha}^\epsilon)=
\be_{(-\alpha)^\prime}^\epsilon=\be_{-\alpha^\prime}^\epsilon$. Hence, 
since $\be_\alpha^\epsilon=-\omega(\be_{-\alpha}^\epsilon)$, we obtain
\[\tau(\be_{\alpha}^\epsilon)=-(\tau \circ \omega)
(\be_{-\alpha}^\epsilon)=-(\omega \circ \tau)
(\be_{-\alpha}^\epsilon)=-\omega(\be_{-\alpha^\prime}^\epsilon)
=\be_{\alpha^\prime}^\epsilon,\]
as required. The assertion about $N_{\alpha,\beta}^\epsilon$ is now
clear since $\tau$ is a homomorphism.
\end{proof}

\begin{abs} \label{auto3} We set $\tfg:= \{x\in \fg\mid \tilde{\tau}(x)=
x\}$; clearly, this is a Lie subalgebra of $\tfg$. For any $i\in I$ we 
denote by $\ovi$ the orbit of $i$ under the permutation $i\mapsto i^\prime$ 
of $I$; let $d_i:=|\ovi|$. Recall from \ref{auto1} that $\tI$ is a (certain) 
set of representatives of those orbits. Let $i\in \tilde{I}$ and $\ovi=
\{i_1,\ldots,i_r\}$ be the orbit of $i$ where $r=d_i$; then we set 
\[\te_i :=e_{i_1}+\ldots +e_{i_r}, \qquad \tf_i :=f_{i_1}+\ldots +f_{i_r},
\qquad \tha_i :=h_{i_1}+\ldots +h_{i_r}.\]
These elements only depend on the orbit of $i$; furthermore,
$\te_i,\tf_i,\tha_i\in \tfg$ for all $i\in \tI$. A straightforward 
computation, using the relations in \ref{eps1} and the definition of 
$\tilde{A}= (\tilde{a}_{ij})_{i,j \in \tI}$ in \ref{auto1}, shows that 
\begin{alignat*}{2}
[\te_i,\tf_i] &=\tha_i \quad&\mbox{and}\quad & [\te_i,\tf_j]=0 \quad
\mbox{for $i,j\in I$ such that $i\neq j$},\\
[\tha_i,\te_j] &=\tilde{a}_{ij}\te_j \quad&\mbox{and}\quad &
[\tha_i,\tf_j]=-\tilde{a}_{ij}\tf_j \quad\mbox{for all $i,j\in \tI$}.
\end{alignat*}
Also note that $\tfh:=\langle  \tha_i \mid i \in \tI\rangle_\Q
\subseteq \tfg$ is an abelian subalgebra (since $[h_i,h_j]=0$ for all 
$i,j\in I$). 
\end{abs}

\begin{abs} \label{auto4} For any $\alpha\in \Phi$, we denote by $\oalpha$ 
the orbit of $\alpha$ under the permutation $\alpha\mapsto \alpha'$ 
of~$\Phi$; again, we have $|\oalpha|\in \{1,d\}$. We set
\[ \textstyle \tbe_\alpha^\epsilon :=\sum_{\beta\in \oalpha}
\be_{\beta}^\epsilon\in \fg.\]
By Proposition~\ref{prop34}, we have $\tbe_\alpha^\epsilon\in \tfg$. Also
note that we certainly have $\tbe_\alpha^\epsilon=\tbe_\beta^\epsilon$ 
for all $\beta\in \oalpha$. So let $\tPhi$ be a set of representatives for
the orbits of $\Phi$ under the permutation $\alpha \mapsto \alpha^\prime$. 
Since $\cB^\epsilon=\{h_i\mid i \in I\} \cup \{\be_\alpha^\epsilon\mid\alpha 
\in \Phi\}$ is a basis of $\fg$, and since $\tau$ is an automorphism of 
finite order which permutes the elements of $\cB^\epsilon$ (see again
Proposition~\ref{prop34}), it easily follows by an elementary argument that 
\[\tcB^\epsilon:=\{\tha_i\mid i \in \tilde{I}\}\cup \{\tbe_\alpha^\epsilon 
\mid \alpha \in \tPhi\}\qquad \mbox{is a basis of $\tfg$}.\]
Note that the elements of $\tcB^\epsilon$ are integral linear combinations
of elements of the basis $\cB^\epsilon$ of $\fg$. Hence, since the complete 
multiplication table of $\fg$ with respect to~$\cB^\epsilon$ has only entries
in~$\Z$, it is already clear that the complete multiplication table of 
$\tfg$ with respect to $\tcB^\epsilon$ will only have entries in~$\Q$.
Now the following facts are standard; see, e.g., De Graaf 
\cite[\S 5.15]{graaf}, Kac \cite[\S 7.9]{kac} or Mitzman 
\cite[\S 3.2]{mitz}. 
\begin{itemize}
\item[(a)] The Lie algebra $\tfg$ is simple and $\tfh\subseteq \tfg$ is 
a Cartan subalgebra. 
\item[(b)] A system of Chevalley generators for $\tfg$ is given by 
$\{\te_i,\tf_i,\tha_i\mid i \in \tI\}$; the corresponding Cartan matrix 
is $\tilde{A}$ (of the type specified in Table~\ref{Mautos}).
\end{itemize}
\end{abs}

We shall deal with the roots for~$\tfg$ in a way that is somewhat different 
from \cite{graaf}, \cite{kac} or \cite{mitz}. This will also lead to an 
explicit description of the co-roots in Corollary~\ref{auto6a} below. 
Note that, in \cite{graaf} and \cite{kac}, the definition of co-roots is
not consistent with the general theory. (See the formulae in 
\cite[(5.18)]{graaf} or \cite[(7.9.3)]{kac}.) 

For $\alpha\in \Phi$ we denote by $\talpha \colon \tfh\rightarrow \C$ the 
restriction of $\alpha$ to the subspace $\tfh\subseteq \fh$. Note that all 
roots in the orbit $\oalpha$ have the same restriction to $\tfh$. (This
immediately follows from the formula $\alpha^\prime(\tau(h))=
\tau_\fh^*(\alpha)(h)=\alpha(h)$.) A priori, it could happen
that two roots have the same restriction to $\tfh$ even if they are
not in the same orbit. But the following result shows that this can 
not happen.

\begin{lem} \label{auto6} The roots of $\tfg$ with respect to $\tfh$ are 
given by the set $\{\talpha\mid \alpha \in \Phi\}\subseteq \tfh^*$, where 
$\talpha=\tbeta$ if and only if $\oalpha=\obeta$. We have $\tfg_{\talpha}=
\langle \tbe_\alpha^\epsilon \rangle_\C$ for all $\alpha\in \tPhi$.  
\end{lem}

\begin{proof} We claim that $[\tha_i,\tbe_\alpha^\epsilon]= \alpha(\tha_i)
\tbe_\alpha^\epsilon$ for $i \in \tI$ and $\alpha\in \Phi$. To see 
this, note again that $\alpha^\prime (\tau(h))=\tau(h)$ for $h\in \fh$ 
(by the definition of $\tau_\fh^*$). Hence, since $\tau(\tha_i)=\tha_i$, 
we have $\alpha^\prime (\tha_i)=\alpha(\tha_i)$, and so $\beta(\tha_i)=
\alpha(\tha_i)$ for all $\beta\in\oalpha$. It follows that
\[ [\tha_i,\tbe_\alpha^\epsilon]=
\sum_{\beta\in \oalpha} [\tha_i,\be_\beta^\epsilon]=
\sum_{\beta\in \oalpha} \beta(\tha_i)\be_\beta^\epsilon=
\alpha(\tha_i) \sum_{\beta\in \oalpha} \be_\beta^\epsilon=
\alpha(\tha_i)\tbe_\beta^\epsilon,\]
as desired. This means that $\tbe_\alpha^\epsilon\in \tfg_{\talpha}=$ 
weight space of $\tfg$ for $\talpha\in \tfh^*$. Since $\tfh=\tfg_0$, we
must have $\talpha\neq 0$; thus, $\talpha$ is a root of $\tfg$. Since 
the weight spaces corresponding to roots are $1$-dimensional, we conclude 
that $\tfg_\talpha=\langle \tbe_\alpha^\epsilon\rangle_\C$. Since 
$\tcB^\epsilon$ is a basis of $\tfg$, it follows that $\tfg=\tfh\oplus 
\bigoplus_{\alpha \in \tPhi} \tfg_\talpha$ is the Cartan decomposition
of $\tfg$; in particular, we must have $\talpha\neq \tbeta$ if 
$\oalpha\neq \obeta$.
\end{proof}

\begin{cor} \label{auto6a} In the above setting, the following holds.
\begin{itemize}
\item[{\rm (a)}] Let $\alpha\in \Phi$ be such that $\alpha\neq 
\alpha^\prime$. Then $\alpha\pm \alpha^\prime\not\in \Phi$ and
$\alpha\pm \alpha^{\prime\prime}\not \in \Phi$.
\item[{\rm (b)}] The co-root corresponding to $\alpha\in \tPhi$ is 
given by $\tha_\alpha^\epsilon:=\sum_{\beta\in \oalpha}h_{\beta} \in
\tfh$. Thus, we have $\tha_\alpha^\epsilon\in [\tfg_\alpha,
\tfg_{-\alpha}]$ and $\talpha(\tha_\alpha^\epsilon)=2$.
\end{itemize}
\end{cor}

\begin{proof} (a) Assume that $\beta:=\alpha+\alpha'\in \Phi$. Since
$\alpha$ and $\alpha^\prime$ have the same restriction to $\tfh$, it
follows that $\tbeta=2\talpha$ is a root of $\tfg$, contradiction 
since the root system of a simple Lie algebra is reduced. Now assume 
that $\gamma:=\alpha-\alpha'\in \Phi$. Then $\tgamma=0$ is a root of
$\tfg$, again a contradiction. The proofs for $\alpha\pm \alpha^{\prime
\prime}$ are analogous.

(b) By Proposition~\ref{prop34}, it is clear that $\tha_\alpha\in\tfg$.
Now let $\beta,\gamma \in \oalpha$. If $\beta\neq \gamma$, then $\beta=
\gamma^\prime$ or $\beta=\gamma^{\prime\prime}$. Hence, (a) implies that 
$\beta-\gamma\not\in \Phi$ and so $[\be_\beta^\epsilon,
\be_{-\gamma}^\epsilon]=0$. This yields 
\[[\tbe_\alpha^\epsilon,\tbe_{-\alpha}^\epsilon]=
\sum_{\beta \in \oalpha} \sum_{\gamma\in \oalpha}
[\be_\beta^\epsilon,\be_{-\gamma}^\epsilon]=
\sum_{\beta \in \oalpha} [\be_\beta^\epsilon,\be_{-\beta}^\epsilon]\\
=\sum_{\beta \in \oalpha} (-1)^{\hgt(\beta)}h_\beta,\]
where the last equality holds by \ref{cb4}.
Since all roots in $\oalpha$ have the same height, the right hand side
equals $(-1)^{\hgt(\alpha)}\tha_\alpha$; thus, $\tha_\alpha\in 
[\tfg_\alpha,\tfg_{-\alpha}]$. Now consider $\talpha(\tha_\alpha^\epsilon)=
\sum_{\beta \in \oalpha} \alpha(h_\beta)$. Let $\beta\in \oalpha$.
If $\beta=\alpha$, then $\alpha(h_\beta)=\alpha(h_\alpha)=2$. If
$\beta\neq \alpha$, then $\beta=\alpha^\prime$ or $\beta=\alpha^{\prime
\prime}$. Assume, if possible, that $\alpha(h_\beta)\neq 0$.
Then the invariance property in \ref{coroot1} implies that 
$\alpha\pm \beta \in \Phi$, that is, $\alpha\pm \alpha^\prime\in \Phi$
or $\alpha\pm \alpha^{\prime\prime}\in \Phi$, contradiction to~(a). 
Hence, $\talpha(\tha_\alpha^\epsilon)=\alpha(h_\alpha)=2$.
\end{proof}

\begin{table}[htbp] \caption{Roots of type $D_4$ (see Example~\ref{typeD4})} 
\label{MtblD4} 
\begin{center} 
$\renewcommand{\arraystretch}{1.1} \begin{array}{ccc} \hline 
\mbox{orbits in $\Phi^+$} & \qquad & \talpha \\\hline
\{1000,0100,0001\} && \talpha_1\\
\{0010\} && \talpha_3\\
\{1010,0110,0011\} && \talpha_1+\talpha_3 \\ 
\{1110,1011,0111\} && 2\talpha_1+\talpha_3\\ 
\{1111\} && 3\talpha_1+\talpha_3 \\
\{1121\} && 3\talpha_1+2\talpha_3\\\hline\end{array}$
\end{center}
\end{table}

\begin{exmp} \label{typeD4} Let $A$ be of type $D_4$, where $I=\{1,2,3,4\}$
and $1^\prime=2$, $2^\prime=4$, $4^\prime=1$, $3^\prime=3$. Let $\{\alpha_1,
\alpha_2,\alpha_3,\alpha_4\}$ be a system of simple roots in $\Phi$. We 
take $\tI=\{1,3\}$; then we have two linearly independent restrictions
$\{\talpha_1,\talpha_3\}$. Table~\ref{MtblD4} shows the orbits of 
$\alpha\mapsto \alpha^\prime$ on the $12$ positive roots in $\Phi$, and 
expressions for $\talpha$ as linear combinations of $\talpha_1,\talpha_3$.
(In that table, we simply write $1011$ instead of $\alpha_1+\alpha_3+
\alpha_4$, for example.) In the second column, we recognise a root system 
of type $G_2$, as expected from Table~\ref{Mautos}. 
\end{exmp}

Here is a key property of the signs $\heta(\alpha,\beta)$ in
Definition~\ref{def1}.

\begin{lem} \label{lemheta} Let $\alpha,\beta\in \Phi$ be such that
$\alpha+\beta\in \Phi$. Let $S(\alpha,\beta)$ be the set of all pairs
$(\alpha_0,\beta_0)$ where $\alpha_0\in \oalpha$, $\beta_0\in \obeta$ and
$\alpha_0+\beta_0\in \Phi$. Then $\heta(\alpha,\beta)=\heta(\alpha_0,
\beta_0)$ for all $(\alpha_0, \beta_0)\in S(\alpha,\beta)$.
\end{lem}

\begin{proof} First note that, by Theorem~\ref{thm1} and 
Proposition~\ref{prop34}, we certainly have 
\begin{equation*}
\heta(\alpha_0,\beta_0)= \heta(\alpha_0^\prime,\beta_0^\prime) \qquad
\mbox{for all $(\alpha_0,\beta_0)\in S(\alpha,\beta)$}.\tag{$*$}
\end{equation*}
Also note that $\talpha_0+\tbeta_0=\talpha+\tbeta$ for all $(\alpha_0,
\beta_0)\in S(\alpha,\beta)$. Hence, by Lemma~\ref{auto6}, 
\begin{equation*}
(\alpha_0,\beta_0)\in S(\alpha,\beta)\quad \Rightarrow\quad \alpha_0+\beta_0
\in \underline{\alpha{+}\beta} \quad \mbox{(orbit of $\alpha+\beta$)}.
\tag{$\dagger$}
\end{equation*} 
Now we distinguish four cases. 

\smallskip
\noindent {\it Case 1}. Assume that $\alpha=\alpha^\prime$. Then 
$\alpha+\beta_0\in \Phi$ for all $\beta_0\in \obeta$. So 
$S(\alpha,\beta\}=\{(\alpha,\beta_0) \mid \beta_0\in \obeta\}$ and 
the desired property follows from~($*$).

\smallskip
\noindent {\it Case 2}. Assume that $\beta=\beta^\prime$. This is 
completely analogous to Case~1. 

\smallskip
\noindent {\it Case 3}. Assume that $\alpha\neq \alpha^\prime$ and $\beta
\neq \beta^\prime$ but $\alpha+\beta=\alpha^\prime+\beta^\prime$. The latter 
condition means that the orbit of $\alpha+\beta$ is just $\{\alpha+\beta\}$. 
Now assume, if possible, that $(\alpha,\beta^\prime)\in S(\alpha,\beta)$.
Then ($\dagger$) would imply $\alpha+\beta^\prime\in \{\alpha+\beta\}$, 
contradiction. Hence, we have $\alpha+\beta^\prime\not\in \Phi$.
Similarly, one sees that $\alpha+ \beta^{\prime\prime}\not\in \Phi$ if 
$d=3$. Consequently, we have $S(\alpha,\beta)=\{(\alpha,\beta), 
(\alpha^\prime, \beta^\prime)\}$ (if $d=2$) or $S(\alpha,\beta)= 
\{(\alpha,\beta), (\alpha^\prime,\beta^\prime), (\alpha^{\prime\prime}, 
\beta^{\prime \prime})\}$ (if $d=3$). So the desired assertion follows 
again from~($*$).

\smallskip
\noindent {\it Case 4}. Assume that $\alpha\neq \alpha^\prime$ and $\beta
\neq \beta^\prime$ and $\alpha+\beta\neq \alpha^\prime+\beta^\prime$.
First let $d=2$. Then $\oalpha=\{\alpha,\alpha^\prime \}$ and $\obeta=
\{\beta,\beta^\prime\}$. If we had $\alpha+\beta^\prime \in\Phi$, then 
$(\dagger$) would imply that $\alpha+\beta^\prime$ belongs to the orbit 
of $\alpha+\beta$,  contradiction because that orbit is $\{\alpha+\beta, 
\alpha^\prime+\beta^\prime\}$. Hence, $\alpha+ \beta^\prime\not\in
\Phi$. Similarly, one sees that $\alpha^\prime+ \beta\not\in \Phi$. 
Consequently, we have $S(\alpha,\beta)=\{(\alpha,\beta), (\alpha^\prime, 
\beta^\prime)\}$ and so, again, the desired property follows from~($*$).

Finally, assume that $d=3$, so we are in the situation of 
Example~\ref{typeD4}. By inspection of Table~\ref{MtblD4}, one sees
that one of $\alpha,\beta,\alpha+\beta$ must be equal to $\pm \alpha_i$
for some $i\in \{1,2,4\}$. If $\alpha=\pm \alpha_i$ or $\beta=\pm \alpha_i$,
then the desired assertion follows from Example~\ref{def1a} and 
Lemma~\ref{heta1}. So assume now that $\alpha+\beta=\pm \alpha_i$. 
Using ($*$) and Lemma~\ref{heta1}, it is enough to consider the case
where $\alpha+\beta=\alpha_1$. An inspection of Table~\ref{MtblD4} gives 
only two possibilities: 
\[ (\alpha,\beta)=(1110,-0110) \qquad\mbox{or} \qquad (\alpha,\beta)= 
(1011,-0011).\]
In the first case, $\oalpha=\{1110,0111,1011\}$ and $\obeta=\{-0110,-0011,
-1010\}$; furthermore, $|S(\alpha,\beta)|=6$. Using the formula in 
Definition~\ref{def1} one finds that $\heta(\alpha_0, \beta_0)=1$ for
all $(\alpha_0,\beta_0)\in S(\alpha, \beta)$. The second case is analogous.
\end{proof}

\begin{thm} \label{thm2} The set $\tcB^\epsilon=\{\tha_i \mid
i \in \tI\}\cup \{\tbe_\alpha^\epsilon\mid \alpha\in \tPhi\}$ is
the $\tilde{\epsilon}$-canonical Chevalley basis of $\tfg$, where
$\tilde{\epsilon}$ is the restriction of $\epsilon$ to~$\tI$. Let 
$\alpha,\beta\in \Phi$ be such that $\talpha+\tbeta$ is a root for
$\tfg$. Replacing $\beta$ by $\beta^\prime$ or $\beta^{\prime\prime}$
if necessary, we may assume without loss of generality that $\alpha+
\beta \in \Phi$. Then we have
\[ [\tbe_\alpha^\epsilon,\tbe_\beta^\epsilon]=\heta(\alpha,\beta)
(\tq_{\alpha,\beta}+1)\tbe_{\alpha+\beta}^\epsilon,\]
where $\tq_{\alpha,\beta}:=\max\{m \in \Z_{\geq 0} \mid \tbeta-m\talpha
\mbox{ is a root for $\tfg$}\}$ and $\heta(\alpha,\beta)=\pm 1$ is
as in Theorem~\ref{thm1}.
\end{thm}

\begin{proof} Let $\alpha,\beta\in \Phi$ be such that $\talpha+\tbeta$ 
is a root for $\tfg$. Then 
\begin{equation*}
[\tbe_\alpha^\epsilon,\tbe_\beta^\epsilon]=\sum_{\alpha_0 \in \oalpha} 
\sum_{\beta_0\in \obeta} [\be_{\alpha_0}^\epsilon,\be_{\beta_0}^\epsilon]
=\sum_{(\alpha_0,\beta_0)\in S(\alpha,\beta)} N_{\alpha_0,\beta_0}^\epsilon
\be_{\alpha_0+ \beta_0}^\epsilon,\tag{a}
\end{equation*}
with $S(\alpha,\beta)$ as in Lemma~\ref{lemheta}. Since $\talpha+\tbeta$ 
is a root for~$\tfg$, the left hand side of~(a) is non-zero. So there must 
exist some $(\alpha_0,\beta_0)\in S(\alpha,\beta)$ such that $\alpha_0+
\beta_0\in \Phi$. Note that then we also have $(\alpha_0^\prime,
\beta_0^\prime)\in S(\alpha, \beta)$ and $(\alpha_0^{\prime\prime},
\beta_0^{\prime\prime})\in S(\alpha,\beta)$. Hence, replacing~$\beta$ 
by $\beta^\prime$ or $\beta^{\prime \prime}$ if necessary, we may assume
without loss of generality that $\alpha+\beta \in \Phi$. In this case, we 
can write 
\[[\tbe_\alpha^\epsilon,\tbe_\beta^\epsilon]=\tilde{N}_{\alpha,
\beta}^\epsilon \tbe_{\alpha+\beta}^\epsilon\qquad \mbox{where}\qquad
\tilde{N}_{\alpha,\beta}^\epsilon\in \C.\] 
As a first step, we show that $\tcB^\epsilon$ is a Chevalley basis of
$\tfg$. (This will determine the absolute value of $\tilde{N}_{\alpha,
\beta}^\epsilon$.) For this purpose, we use the criterion in \ref{chevb} 
and consider the automorphism $\omega \colon \fg \rightarrow \fg$. As already
noted in the proof of Proposition~\ref{prop34}, the maps $\omega$ and $\tau$ 
commute with each other. Hence, $\omega(\tfg)=\tfg$ and so $\omega$ restricts
to an automorphism of $\tfg$ which we denote by the same symbol. One
immediately checks that $\omega(\te_i)=\tf_i$, $\omega(\tf_i)=\te_i$ and
$\omega(\tha_i)=-\tha_i$ for $i \in \tI$. So, by \ref{cb4}, we have
$\omega(\be_\alpha^\epsilon)= -\be_{-\alpha}^\epsilon$ for all 
$\alpha\in \Phi$. Hence, we obtain
\[ \omega(\tbe_\alpha^\epsilon)=\sum_{\beta \in \oalpha}
\omega(\be_\alpha^\epsilon)=\sum_{\beta \in \oalpha} 
-\be_{-\alpha}^\epsilon=-\tbe_{-\alpha}^\epsilon.\]
By Corollary~\ref{auto6a}(b) (and its proof), we also have 
$[\tbe_\alpha^\epsilon, \tbe_{-\alpha}^\epsilon]=(-1)^{\hgt(\alpha)}
\tha_\alpha$. Hence, the conditions in \ref{chevb} are satisfied and
so $(\tilde{N}_{\alpha, \beta}^\epsilon)^2=\pm (\tq_{\alpha,\beta}+1)^2$ 
whenever $\alpha,\beta, \alpha+\beta \in \Phi$. Finally, we already noted 
in \ref{auto4} that the complete multiplication table of~$\tfg$ with
respect to~$\tcB^\epsilon$ has only entries in~$\Q$; in particular, 
$\tilde{N}_{\alpha,\beta}^\epsilon \in \Q$ and, hence, $\tilde{N}_{\alpha,
\beta}^\epsilon=\pm (\tq_{\alpha,\beta}+1)$, as required. (The above argument
essentially appeared already in the proof of \cite[Theorem~3.2.26]{mitz}.)

Now let again $\alpha,\beta\in \tPhi$ be arbitrary such that $\talpha+
\tbeta$ is a root for $\tfg$. As before, we may assume without loss of 
generality that $\alpha+\beta\in \Phi$. Since $\tcB^\epsilon$ is a Chevalley
basis, there is a sign $\xi=\pm 1$ such that $\tilde{N}_{\alpha,
\beta}^\epsilon=\xi (\tq_{\alpha,\beta}+1)$. Then 
\begin{equation*}
[\tbe_\alpha^\epsilon,\tbe_\beta^\epsilon]=\tilde{N}_{\alpha,\beta}^\epsilon
\tbe_{\alpha+\beta}^\epsilon=\xi (\tq_{\alpha,\beta}+1)\sum_{\gamma\in 
\underline{\alpha{+}\beta}} \be_{\gamma}^\epsilon. \tag{b} 
\end{equation*}
In order to determine $\xi=\pm 1$, we compare the two expressions (a) 
and~(b). All coefficients of basis elements in those expressions are 
in~$\Z$, and~(b) shows that they all have the same sign, given by~$\xi$. On 
the other hand, by Theorem~\ref{thm1} and Lemma~\ref{lemheta}, all 
coefficients $N_{\alpha_0, \beta_0}$ in the sum in~(a) are equal to
$\heta(\alpha,\beta)$. Hence, there are no cancellations in~(a) and so 
$\xi=\heta(\alpha,\beta)$; furthermore, 
\begin{equation*}
\tq_{\alpha,\beta}+1\,=\, |\{(\alpha_0,\beta_0) \in S(\alpha,\beta)\mid
\alpha_0+ \beta_0=\alpha+\beta\}|. \tag{c}
\end{equation*}
Thus, we have shown that 
\[ [\tbe_\alpha^\epsilon,\tbe_\beta^\epsilon]=\heta(\alpha,\beta)
(\tq_{\alpha,\beta}+1)\tbe_{\alpha+\beta}^\epsilon\qquad\mbox{whenever
$\alpha,\beta,\alpha+\beta\in \Phi$}.\]
Finally, we show that $\tcB^\epsilon$ is the $\tilde{\epsilon}$-canonical 
Chevalley basis of $\tfg$; recall the conditions from \ref{cb4}. First let 
$i \in \tI$. Then $\tbe_{\alpha_i}^\epsilon=\sum_{j \in \ovi} 
\be_{\alpha_i}^\epsilon= \sum_{j \in \ovi} \epsilon(j) e_j$. Since 
$\epsilon$ is constant on the orbits of the permutation 
$i\mapsto i^\prime$, we conclude that $\tbe_{\alpha_i}^\epsilon=
\tilde{\epsilon}(i)\te_i$. The argument for $\tbe_{-\alpha_i}^\epsilon$ 
is completely analogous. Now let also $\alpha\in \tPhi$ be such that 
$\talpha_i+\talpha$ is a root for $\tfg$; as before, we may assume without
loss of generality that $\alpha_i+\alpha\in \Phi$. By the above formula
we know that 
\[[\tilde{e}_i,\tbe_\alpha^\epsilon]=\epsilon(i)[\tbe_{\alpha_i}^\epsilon,
\tbe_\alpha^\epsilon]=\epsilon(i)\heta(\alpha_i,\alpha)
(\tq_{\alpha_i,\alpha}+1) \tbe_{\alpha_i+\alpha}^\epsilon=
(\tq_{\alpha_i,\alpha}+1) \tbe_{\alpha_i+\alpha}^\epsilon,\]
where the last equality holds by Example~\ref{def1a}. The argument 
for $[\tilde{f}_i,\tbe_\alpha^\epsilon]$ is completely analogous (assuming
that $\talpha-\talpha_i$ is a root for $\tfg$).
\end{proof}

\begin{rem} \label{thm2a} Let $\alpha,\beta,\alpha+\beta\in \Phi$.
By the detailed discussion in the proof of \cite[Lemma~5.15.9]{graaf},
one can explicity work out the number $\tq_{\alpha,\beta}$ according to 
formula~(c) in the above proof. Examples: If $\alpha=\alpha^\prime$ or 
$\beta=\beta^\prime$, then $\tq_{\alpha,\beta}=0$. Now assume that 
$\alpha\neq \alpha^\prime$ and $\beta\neq \beta^\prime$. If $\alpha+\beta=
\alpha^\prime+\beta^\prime$, then $\tq_{\alpha,\beta}=d-1$ where~$d$ is 
the order of~$\tau$; if $\alpha+\beta\neq \alpha^\prime+ \beta^\prime$ 
and $d=2$, then $\tq_{\alpha,\beta}=0$. This actually covers all cases 
for $d=2$. If $d=3$, then we are in the situation of Example~\ref{typeD4}; 
the structure constants for $\tfg$ in this case are explicitly listed 
in Table~1 (p.~3246) of~\cite{mylie}.
\end{rem}



\begin{thebibliography}{131}
\bibitem{bour1}
{\sc N.~Bourbaki}, {\it Groupes et alg{\`e}bres de Lie, chap. 4, 5 et 6},
Hermann, Paris, 1968. 

\bibitem{bour2}
{\sc N.~Bourbaki}, {\it Groupes et alg{\`e}bres de Lie, chap. 7 et 8},
Hermann, Paris, 1975.

\bibitem{C1}
{\sc R. W. Carter}, \textit{Simple groups of Lie type}, Wiley, New 
York, 1972; reprinted 1989 as Wiley Classics Library Edition.

\bibitem{chev}
{\sc C.~Chevalley}, Sur certains groupes simples, Tohoku Math. J. 
\textbf{7} (1955), 14--66. 

\bibitem{graaf}
{\sc W. A. De Graaf}, \textit{Lie algebras: Theory and Algorithms},
North-Holland Mathematical Library, vol.~56, Elsevier 2000.

\bibitem{mylie}
{\sc M. Geck}, On the construction of semisimple Lie algebras and 
Chevalley groups, Proc. Amer. Math. Soc.  {\bf 145} (2017), 3233--3247.

\bibitem{chevlie}
{\sc M. Geck}, {\sf ChevLie}: Constructing Lie algebras and Chevalley
groups, J. Softw. Algebra Geom. {\bf 10} (2020), 41--49; see also
\url{https://github.com/geckmf/ChevLie.jl}.

\bibitem{H}
{\sc J. E. Humphreys}, {\it Introduction to Lie algebras and representation
theory}, Graduate Texts in Mathematics, 9, Springer-Verlag,
New York-Berlin, 1972.

\bibitem{kac}
{\sc V. Kac}, \textit{Infinite dimensional Lie algebras}, Cambridge
University Press, 1985.

\bibitem{Lang}
{\sc A. Lang}, Kanonische Strukturkonstanten in einfachen 
Lie-Algebren, Master-Arbeit, Universit\"at Stuttgart, 2023.

\bibitem{L1}
{\sc G.~Lusztig}, \emph{Introduction to quantum groups}, Modern 
Birkh\"{a}user Classics, Birkh\"{a}user/Springer, New York, 2010, 
Reprint of the 1994 edition. 

\bibitem{L5}
{\sc G. Lusztig}, The canonical basis of the quantum adjoint 
representation, J. Comb. Alg. {\bf 1} (2017), 45--57.

\bibitem{mitz}
{\sc D. Mitzman}, \emph{Integral bases for affine {L}ie algebras
and their universal enveloping algebras}, Contemp. Math., vol.~40,
Amer. Math. Soc., Providence, RI, 1985.

\bibitem{St}
R.~Steinberg, \textit{Lectures on Chevalley groups}. Mimeographed notes,
Department of Math., Yale University, 1967. Now available as vol.~66 of
the University Lecture Series, Amer. Math. Soc., Providence, R.I., 2016.
\end{thebibliography}
\end{document}